\documentclass[12pt]{article}
\usepackage{amsthm,amssymb,latexsym,amsmath,subfigure}
\usepackage{graphicx,color,xcolor}
\usepackage{authblk}
\usepackage{blindtext}
\usepackage{hyperref}

\newtheorem{theorem}{Theorem}
\newtheorem{proposition}{Proposition}
\newtheorem{corollary}{Corollary}
\newtheorem{definition}{Definition}

\newtheorem{remark}{Remark}

\newtheorem{lemma}{Lemma}

\newcommand{\R}{{\mathbb{R}}}
\makeatletter
\renewcommand\AB@affilsepx{ --- \protect\Affilfont}
\makeatother

\title{Commutator Method for Averaging Lemmas}

\date{\vspace{-5ex}}

\author{Pierre-Emmanuel Jabin\thanks{Center for Scientific Computation and Mathematical Modeling (CSCAMM) and Department of Mathematics, University of Maryland, College Park, MD 20742, USA. P.-E. Jabin is partially supported by NSF DMS Grant 161453, 1908739,
and NSF Grant RNMS (Ki-Net) 1107444. Email: \texttt{pjabin@cscamm.umd.edu.}}, Hsin-Yi Lin\thanks{CSCAMM and Department of Mathematics, University of Maryland, College Park, MD 20742, USA. Email: \texttt{hylin@cscamm.umd.edu.}}, Eitan Tadmor\thanks{CSCAMM, Department of Mathematics, and Institute for Physical Sciences \& Technology (IPST), University of Maryland, College Park, MD 20742, USA. Research of E. Tadmor was supported in part by NSF grants DMS16-13911, RNMS11-07444 (KI-Net) and ONR grant N00014-1812465. Email: \texttt{tadmor@cscamm.umd.edu.}}}

\begin{document}
\maketitle
\begin{abstract}
We introduce a commutator method with multipliers to prove averaging lemmas, the regularizing effect for the velocity average of solutions for kinetic equations. This method requires only elementary techniques in Fourier analysis and shows a new range of assumptions that are sufficient for the velocity average to be in $L^2([0,T],H^{1/2}_x).$ This result not only shows an interesting connection between averaging lemmas and local smoothing property of dispersive equations, but also provide a direct proof for the regularizing effect for the measure-valued solutions of scalar conservation laws in space dimension one.

\end{abstract}

\section{Introduction}

\subsection{Brief overview for averaging lemmas}

Our goal of this paper is to introduce the \textbf{commutator method} for kinetic transport equations:
\begin{equation}\label{transport}
\varepsilon \partial_t f+a(v)\cdot\nabla_x f=(-\Delta_v)^{\alpha/2} g,
\end{equation}
where $\varepsilon>0,$ $\alpha\geq 0$, $a:\mathbb{R}_v^n \rightarrow\mathbb{R}^n$ and $g:\mathbb{R}_t \times \mathbb{R}^n_v \times \mathbb{R}_x^n \rightarrow \mathbb{R}$ are given functions. $\varepsilon$ is the macroscopic scale normally introduced when a hydrodynamics limit is considered. The nonlinear coefficients $a(v)$ in this setting appears in kinetic formulation of scalar conservation law, and also in kinetic models under relativistic and quantum setting \cite{Escobedo}, \cite{Golse5}.

We shall utilize this method as a new approach to derive \textbf{averaging lemmas}, which state that by taking average in microscopic $v$ variable, the velocity average of $f$ 
$$
\rho_\phi(t,x):=\int f(t,x,v)\phi(v)\, dv, \,\,\,\,\,\,\text{$\phi\in L_c^\infty,$}
$$
has better regularity than $f$ and $g$ in $x$ variable, where $L^\infty_c$ is the space containing all the bounded and compactly supported functions. There is a vast literature of averaging lemmas, and here we only mention few of them that are relatively closer to our discussion. 
This type of results is famous for getting compactness for the Vlasov-Maxwell system \cite{DiPerna3}, renormalized solutions  \cite{DiPerna4} and hydrodynamic limits for the Boltzmann equation \cite{Golse3}, and the convergence of the renormalized solutions to the semiconductor Boltzmann-Poisson system \cite{Masmoudi}. It also contributes to the regularizing effect of solutions wherever the kinetic formulations exist, such as the isentropic gas dynamics \cite{Lions3}, Ginzburg-Landau model \cite{Jabin1}, and scalar conservation laws \cite{Lions2}.

Classical averaging lemmas were first introduced independently in  \cite{Agoshkov} and \cite{Golse1} under $L^2$ setting. The derivation in \cite{Golse1} involves decomposition in Fourier space according to the order of $a(v)\cdot \xi$, and controlling the singular part $|a(v)\cdot \xi|<c$ with the non-degeneracy condition. Combining with interpolation arguments, it was later extended to general $L^p$, $1<p<\infty$ by \cite{Bezard} and \cite{DiPerna5}. It was followed by the optimal Besov results proved in \cite{DeVore} by using wavelet decomposition. The regularity for the $L^p$ case is further improved in the one-dimensional case, precisely from $\frac{1}{p}$ to $1-\frac{1}{p}$ when $p> 2$ by \cite{Arsenio3}, with dispersive property and dyadic decomposition.

Averaging lemmas under different conditions on $f$ and $g$ were  further discussed. For instance, in \cite{Westdickenberg} the author considered $f$ and $g$ in the same Besov space in $x$ but can have different integrability in $v$. The results for general mixed norms assumptions were obtained in \cite{Jabin3} \cite{Jabin4}. Their work inspired \cite{Arsenio2} to consider the case when $f$ and $g$ have less integrability in $x$ than $v$. Except for the explorations in the direction of general conditions, averaging results for a larger class of operators in the form of $a(v)\cdot\nabla_x-\nabla_x^{\perp}\cdot b(v) \nabla_x$ were acquired by \cite{Tadmor}. They presented several applications for their results and especially, they improved the regularity of solutions for scalar conservation laws.

The limiting $L^1$ case for classical averaging lemmas in general is not true, and a counterexample was given in \cite{Golse1}. However, $L^1$ compactness can be proved with equi-integrability in only $v$ variable \cite{Golse6}, and was extended to more general transport equations in \cite{Arsenio1} and \cite{Daniel}.

\subsection{Commutator method with multiplier technique}

In this work we use commutator method with multipliers to transform the dispersion of transport operator in Fourier space into gain of regularity in $x$ variable. Let us introduce the commutator method in a general setting, and narrow down to our case shortly. Assume
$$
\varepsilon\partial_t f + B f= g,
$$
where $B$ is a skew-adjoint operator, $\varepsilon\leq 1$ and $g$ are given. For a time-independent operator $Q$, we consider 
$$
\varepsilon\partial_t \int f\, \overline{Qf}\, dx\,dv =\int [B,Q] \,  f \bar{f} \, dx\,dv+\int g \, \overline{Qf}\, dx\,dv+\int f \, \overline{Qg}\, dx\,dv
$$
And by fundamental theorem of calculus we have 
\begin{equation}\label{bound}
\begin{split}
Re\int_0^T [B,Q] \,  f \bar{f} \, dx\,dv\, dt \leq &\sup_{t=0,T}\left|\int f\, \overline{Qf}\, dx\,dv \right|+\left|\int g \, \overline{Qf}\, dx\,dv\, dt\right|\\
&    +\left|\int f \, \overline{Qg}\, dx\,dv\, dt\right|.
\end{split}
\end{equation}
The idea is to find $Q$, bounded in some $L^p$ spaces, such that the commutator of $B$ and $Q$, $[B,Q],$ is positive-definite and gain extra derivatives. Hence by applying these conditions on (\ref{bound}) we get a desired bound on $f$.

This method was used for example by taking $B$ to be of Schr\"{o}dinger type, where the commutator appear naturally from the Hamilton vector field. Roughly speaking it involves constructing a proper symbol, which corresponds to $Q,$ such that the Poisson bracket implies a spacetime bound on $f$ by G\r{a}rding's inequality. See for example \cite{Colliander}, \cite{Doi}, \cite{Kajitani} and \cite{Staffilani}.

\bigskip

In this paper we fix $B$ to be the kinetic transport operator, 
\begin{equation}\label{smooth}
\varepsilon \partial_t f+a(v)\cdot\nabla_x f=g,
\end{equation}
and $Q$ is a bounded multiplier operator. That is, we consider
$$
\mathcal{F}_{\xi,\zeta}(Qf):=m(\xi,\zeta)\mathcal{F}_{\xi,\zeta}(f),
$$
where $m$ is bounded. So there is a tempered distribution $K(x,v)$ such that $Qf=K\star_{x,v}f$ with $\mathcal{F}_{\xi,\zeta}(K)=m$. In this case the commutator becomes
\[
\begin{split}
&\int [a(v)\cdot \nabla_x, K\star_{x,v}]f\bar{f}\, dx\, dv\\
&\ =\int (a(v)-a(w))\cdot \nabla_{x}K(x-y,v-w)f(y,w)\, dy\, dw f(x,v)\, dx\, dv.
\end{split}
\]
When $a(v)=v$, it is simply the quadratic form with the multiplier $\xi\cdot\nabla_\zeta m.$ We shall take an advantage of this simple formula and show that the velocity average of $f$ would gain regularity $1/2$ in $x$ when $a(v)=v$, and $g$ is not singular.

The multiplier we select for this purpose is
$$
m_0(\xi,\zeta)=\frac{\xi}{|\xi|}\cdot \frac{\zeta}{(1+|\zeta|^2)^{1/2}},
$$ 
and the corresponding kernel 
$$
K_0=R\cdot\nabla_v G^n_1,
$$
where $R$ is the Riesz potential and $G^n_1$ is the Bessel potential of order $1$ in dimension $n$.
With this choice by Plancherel identity,
\begin{align*}
&\int [v\cdot\nabla_x, K_0 \star_{x,v}] f \bar{f} \, dx\, dv \, dt= \int \xi\cdot\nabla_\zeta m_0 |\hat{f}|^2 \, d\xi \, d\zeta \, dt\\
&\quad=\int\int\left[\frac{1}{(1+|\zeta|^2)^{1/2}}-\frac{\left|\frac{\xi}{|\xi|}\cdot\zeta\right|^2}{(1+|\zeta|^2)^{3/2}}\right]|\xi\|\hat{f}|^2\, d\zeta \, d\xi\,dt\\
&\quad\geq \int\int \frac{|\xi|}{(1+|\zeta|^2)^{3/2}}|\hat{f}|^2\, d\zeta \, d\xi\,dt=\|f\|^2_{L^2\left([0,T], H^{1/2}(\mathbb{R}^n_x, H^{-3/2}(\mathbb{R}^n_v)\right)}.
\end{align*} 

From classical Fourier theory (see for example \cite{Stein}), $K_0$ is bounded on $L^p$ spaces for all $1<p<\infty.$ With this the right hand side of (\ref{bound}) is bounded as long as $f$ is in $L^\infty\left([0,T], L^2(\mathbb{R}^n_x\times\mathbb{R}^n_v)\right)$ and the dual space of $g$. For convenience, let us denote the conjugate index of $p$ by $p',$ that is, $\frac{1}{p}+\frac{1}{p'}=1.$ From the discussion above, we have shown:
\begin{theorem}\label{theorem1}
	Let $\varepsilon\leq 1.$ If $f\in L^\infty\left([0,T], (L^{2}\cap L^{p})(\mathbb{R}^{n}_x\times \mathbb{R}^{n}_v)\right)$ solves (\ref{smooth}) with $a(v)=v$ for
	some $g\in L^1\left([0,T], L^{p'}(\mathbb{R}^{n}_x\times \mathbb{R}^{n}_v)\right)$, where $1< p<\infty$,  then for all $\phi \in H^{3/2}(\mathbb{R}^n_v),$ $\rho_\phi\in L^2\left([0,T], H^{1/2}(\mathbb{R}^n_x)\right),$ and
	\[
	\|f\|^2_{L^2_t H^{1/2}_x H^{-3/2}_v} \leq C \left(\|f\|^2_{L^\infty_t L^{p}_{x,v}}+ \|g\|_{L^1_t L^{p'}_{x,v}}^2\right),
	\]
	where $C$ is independent of $\varepsilon.$
\end{theorem}

\begin{remark}
	By Wigner transform, this result with $p=2$ connects to the local smoothing effect for Schr\"{o}dinger equation. 
\end{remark}

\begin{remark}
	The exchange of regularity between $x$ and $v$ variables is visible through the calculation of commutator, which shares its similarity with the hypoellipticity phenomenon.  Very roughly speaking, it is a phenomenon that the degenerate directions can be recovered by commutators, which was developed systematically by H\"{o}rmander \cite{Hormander} for Fokker-Planck type of operators.  For the hypoellipticity of kinetic transport equations we refer to \cite{Bouchut2}. 
	
	The difference here is that we added a homogeneous zero multiplier $m_0$ as a buffer, which takes on the impact from the transport operator. So the request for extra regularity in $v$ goes to the test function $\phi$, unlike the results in \cite{Bouchut2}, which asked for extra regularity in $v$ for $f$.
\end{remark}

Notice here the requirement of test functions can be adapted to the  $L^\infty_c$, same as classical averaging results. This is because  the product $f\phi$ with $\phi\in L^\infty_c(\mathbb{R}^n_v)$ still satisfies the kinetic transport equation, and the same procedure would give $f\phi \in L^2\left([0,T], H^{1/2}(\mathbb{R}^n_x, H^{-3/2}_v(\mathbb{R}^n_v))\right).$ Now because of the compact support of the integration, we can take a smooth function identically one inside the integral domain. Our main results will require the test functions to be in $L_c^\infty,$ and this argument can be found later in the proof of Theorem \ref{theorem2} in Section \ref{section4}.

Our setting is reminiscent of the multiplier method in \cite{Gasser}. It was used to prove moment and trace lemmas for kinetic equations. For them, the dispersive nature of solutions was acquired by integrating along characteristics in physical space, while here we utilize the technique in frequency domain and so it results in gain of regularity.

For the rest of this paper we are going to extend this method to the general transport equation (\ref{transport}) with the variable coefficient $a(v)$ and a singular source term $(-\Delta_v)^{\alpha/2}g$, which introduce difficult technical issues. The commutator method pairing with $m_0$ will be the main mechanism for our proofs. The advantage of this approach is that the integrability of $f$ and $g$ can be of assistance to each other. This is the feature that distinguishes our results from others in the literature, and provides averaging results for a new type of mixed integrability assumptions, which fits nicely for the conditions that the kinetic formulation of scalar conservation law naturally attain.

This paper is organized as follows. We shall present our main theorems in Section \ref{section2}, and an example of application to scalar conservation laws in Section \ref{section3}. Finally proofs of theorems are in Section \ref{section4}. 

\section{Main results}\label{section2}
\subsection{Our main velocity averaging result}
We present averaging lemmas for (\ref{transport}) derived by the commutator method. To have dispersion in Fourier space for the kinetic transport operator $a(v)\cdot\nabla_x$, one need conditions on the variable coefficients $a(v).$ Indeed, there is no gain of regularity if $a$ is only constant for example.

In this section, we assume $a(v)\in Lip(\mathbb{R}^n)$ with conditions: 
\begin{equation}\label{condition}
\text{$a(v)$ one-to-one, } \text{ and } J_{a^{-1}}\in L^\gamma,
\end{equation}
where $J_{a^{-1}}=det(D a^{-1}).$
The assumptions quantify the nonlinearity of $a(v)$ with index $\gamma$, and allow us to control the integrability of functions after the change of variables $v\mapsto w=a(v).$

Our proof involves regularization of equation (\ref{transport}) through various embeddings. The interaction between embedding and the singular term $(-\Delta_v)^{\alpha/2}g$ will affect the resulting gain of regularity, and this introduce several exponents and indices in the formulas which we collect below,
\begin{equation}\label{d3}
d_1=\max{ \left\{n\left(\frac{1}{p_2}+\frac{1}{q_2}-\ell \right) ,0  \right\}   } , \quad d_2=\max{ \left\{n\left(\frac{2}{p_2}-\ell \right) ,0  \right\}   },\ \ell=\frac{\gamma-2}{\gamma-1},
\end{equation}
\begin{equation}\label{d2}
d_3=\max{ \left\{n\left(\frac{1}{p_1}+\frac{1}{q_1}-1\right) ,0  \right\}   }, \,\,\,\,\,d_4=\max{ \left\{n\left(\frac{2}{p_1}-1\right) ,0  \right\}   }.
\end{equation}

$$
\,
$$

Our result is as follows:
\begin{theorem}\label{theorem2}
	Given $\alpha\geq 0$, $T>0$ and $0<\varepsilon\leq 1.$ Let $a \in Lip(\mathbb{R}^n)$ satisfy (\ref{condition}) with $\gamma \geq 2$. Let $f\in L^\infty([0,T], L^{p_1}(\mathbb{R}^n_x, L^{p_2}(\mathbb{R}^n_v)))$ solve (\ref{transport}) for some $g\in L^1([0,T], L^{q_1}(\mathbb{R}^n_x , L^{q_2}(\mathbb{R}^n_v)))$, with $p_1,\,p_2,\, q_1, \, q_2 \in [1,\infty]$. Then for any ball $B_R(x_0)\subset\mathbb{R}^n_x$ and $\phi \in C^{\infty}_c(\mathbb{R}^n_v)$, one has that 
	$\rho_{\phi}(t,x)\in L^2([0,T], H^{s}(B_R(x_0))) $ 
	for all $s<S,$ with
	$$
	\|\rho_\phi\|^2_{L^2([0,T], H^{s}_x(B_R(x_0)))}\leq C\left(\|f\|_{L^\infty([0,T], L^{p_1}(\mathbb{R}^n_x, L^{p_2}(\mathbb{R}^n_v)))}^2+\|g\|^2_{L^1([0,T], L^{q_1}(\mathbb{R}^n_x, L^{q_2}(\mathbb{R}^n_v)))}\right),
	$$
	where $ S=\frac{1}{2}\left\{(1-d_2)\theta-d_4\right\} $
	with
	$
	\theta=\left[\min\left\{\frac{1-(d_3-d_4)}{\alpha+1+(d_1-d_2)},1  \right\}  \right],
	$
	where $d_i$ are defined in (\ref{d3}) and (\ref{d2}) for $i=1,2,3,4$ and $C$ only depends on $R$, $p_1$, $q_1$ and $Lip(a).$
	\label{maintheorem}
\end{theorem}

\begin{remark}
	The restriction $\gamma \geq 2$ can be relaxed, but with a different formula for $S=\frac{1}{2}\left\{\left[1-n\left(\frac{2}{p_2}+\frac{2}{\gamma}-1\right)\right]\tilde{\theta}-d_4\right\}$ when $1\leq \gamma < 2$, where $\tilde{\theta}=\min\left\{\frac{1-(d_3-d_4)}{\alpha+1+n\left(\frac{1}{q_2}-\frac{1}{p_2}\right)},1  \right\}$.
\end{remark}

\begin{remark}\label{remark4}
	If $f\in L^\infty([0,T], B^0_{p_1,2}(\mathbb{R}^n_x, L^{p_2}(\mathbb{R}^n_v)))$ and $g\in L^1([0,T], B^0_{q_1,2}(\mathbb{R}^n_x, L^{q_2}(\mathbb{R}^n_v))),$ the end point $s=S$ can be included when $p_1,p_2,q_1,q_2\in(1,\infty).$
\end{remark}

\begin{remark}
	Because of the quadratic form in our method, our result always bounds the velocity average in $L^2$, and the bound has the same weight on the norms of $f$ and $g$, independent of $p_1, p_2, q_1, q_2$.
\end{remark}

When $a(v)=v$, one has that $\gamma=\infty.$ In this case, we have a simpler formula for Theorem \ref{theorem2} when $f$ and $g$ are in the dual space of each other:

\begin{corollary}\label{corollary1}
	Given $\alpha\geq 0$, $T>0$ and $0<\varepsilon\leq 1.$ If $f$ belongs to the space $L^\infty\left([0,T],  L^{p_1}(\mathbb{R}^n_x , L^{p_2}(\mathbb{R}^n_v))\right)$ and solves (\ref{transport}) with $a(v)=v$ for
	some\\ $g\in L^1\left([0,T], L^{p_1'}(\mathbb{R}^n_x , L^{p_2'}(\mathbb{R}^n_v))\right)$, where $p_1,p_2\in [2,\infty].$ Then for any $\phi \in C^{\infty}_c(\mathbb{R}^n_v)$, $\rho_\phi\in L^2\left([0,T], H^{s}(\mathbb{R}^n_x)\right)$ for all $s< \frac{1}{2(\alpha+1)} .$
\end{corollary}

Here we conclude with some relations between our result and previous literature.
\begin{itemize}
	\item First of all, let us point out that our velocity averaging result is independent of small $\varepsilon.$ This could have applications to the compactness of solutions for rescaled kinetic equations, which frequently appear in the discussions of hydrodynamic limits. For more in this direction we refer to for example \cite{Golse7} and \cite{Saint}.
	
	Moreover, since our argument doesn't perform a Fourier transform in time variable, this method has possible extensions for time discretized kinetic equations or stochastic cases. 
\end{itemize}

As there is already a huge literature on averaging lemmas, and under some situations the results were proven optimal,  we would like to give the readers an idea on when our method becomes effective, and what are the potential advantages our result could provide.

For the rest of this subsection, we will compare the regularity in $x$ of our result, with the theorems in \cite{Arsenio3}, \cite{DiPerna5} and \cite{Westdickenberg}. Because our resulting space has a different integrability from previous results except for the $L^2$ case, our method may render a more appropriate tool under certain circumstances. We will also point out the regions where one theorem can imply the other, through embedding or interpolation. The interpolation is applied between the resulting space of $\rho_\phi$ and the assumption  space of $f$, because $\rho_\phi $ has the same integrability in $x$ as $f$.

Notice some theorems we quote here apply to more general conditions in the original statements, but for simplicity we shall only state the parts that concern our discussion, and restrict to the special case $a(v)=v$. We also assume for convenience that $f$ and $g$ are compactly supported in $x$ and $v$, and $\phi\in C^\infty_c$ for this entire discussion.

Let us begin with the classical averaging result in \cite{DiPerna5},  where the different integrabilities for $f$ and $g$ and $\alpha>0$ are available.
\begin{theorem}\label{theorem3}\cite{DiPerna5}
	If $f\in L^{p}(\mathbb{R}_t\times \mathbb{R}^n_x \times \mathbb{R}^n_v)$ and $g\in L^{q}(\mathbb{R}_t\times \mathbb{R}^n_x \times \mathbb{R}^n_v)$, satisfying (\ref{transport}) with $a(v)=v$, then $\rho_\phi\in B^{s}_{r,\infty}(\mathbb{R}_t \times \mathbb{R}^n_x)$ where $s=\frac{1}{\bar{p}}\left(\alpha+\frac{1}{\bar{p}}+\frac{1}{\underline{q}}\right)^{-1}$, $\bar{p}=\max\left\{p,p'\right\}$, $\underline{q}=\min\left\{q,q'\right\}$, and $\frac{1}{r}=\frac{s}{q}+\frac{1-s}{p}$. Moreover, if $p=q\in (1,\infty),$ $\rho_\phi\in B^s_{r,t}(\mathbb{R}_t \times \mathbb{R}^n_x)$ where $t=\max\left\{p,2\right\}.$
	\label{DiPerna-Lions}
\end{theorem}

Under the assumption of Theorem \ref{theorem3}, we start our discussions for the cases when $p=q.$
\begin{itemize}
	
	\item \textit{When $p=q=2$, both Theorem \ref{theorem2} and \ref{theorem3} reach the same regularity $H^{\frac{1}{2(1+\alpha)}}.$} 
	
	\item \textit{When $p=q\in (1,2),$ the result by Theorem \ref{theorem3} implies Theorem \ref{theorem2}:} 
	
	Indeed, Theorem \ref{theorem3} reaches $ B^{\frac{1}{p'(1+\alpha)}}_{p,2},$ while Theorem \ref{theorem2} gives $H^{s}$ for all $s<S=\frac{1}{2(1+\alpha)}\left[1-n(2+\alpha)\left(\frac{2}{p}-1\right)\right].$ By embedding theorem for Besov spaces, $B^{\frac{1}{p'(1+\alpha)}}_{p,2}\subset H^{\tilde{s}}$ with $\tilde{s}=\frac{1}{p'(1+\alpha)}+n\left(\frac{1}{2}-\frac{1}{p}\right)$, which is larger or equal to $S$ for all $n\geq 1$ and $p<2$.
	
	\item \textit{When $p=q\in (2,\infty),$ the result by Theorem \ref{theorem2} has more differentiability but less integrability than Theorem \ref{theorem3}. And when $n=1$ and $\alpha=0,$ Theorem \ref{theorem2} implies Theorem \ref{theorem3}:}
	
	Theorem \ref{theorem3} reaches $ B^{\frac{1}{p(1+\alpha)}}_{p,p},$ while Theorem \ref{theorem2} have $H^{\frac{1}{2(1+\alpha)}}.$ Our result has more differentiability but less integrability as $p>2$. By embedding $H^{\frac{1}{2(1+\alpha)}}\subset B^{\tilde{s}}_{p,2},$ where $\tilde{s}=\frac{1}{2(1+\alpha)}+n\left(\frac{1}{p}-\frac{1}{2}\right).$ And $\tilde{s}< \frac{1}{p(1+\alpha)}$ except when $n=1$ and $\alpha=0,$ where the equality holds. 
\end{itemize}

Because of the quadratic form in our method, one sees the more favorable type of conditions for our method is when $p\geq 2$ and $\frac{1}{p}+\frac{1}{q}=1.$ We therefore compare Theorem \ref{theorem2} and \ref{theorem3} under this assumption: 
\begin{itemize}
	\item \textit{Under the assumption of Theorem \ref{theorem3} with $\frac{1}{p}+\frac{1}{q}=1$ and $p\in(2,\infty)$, the result by Theorem \ref{theorem2} has more differentiability but less integrability in $x$. Moreover, Theorem \ref{theorem2} implies Theorem \ref{theorem3} when $\alpha=0$ by interpolation, or when $0\leq \alpha< \frac{1}{n}$ and $2<p<\frac{2n}{n(1+\alpha)-1}$ by embedding:}
	
	Under these conditions, Theorem \ref{theorem2} results in $  H^{\frac{1}{2(1+\alpha)}}(\mathbb{R}^n_x)$, while Theorem \ref{theorem3} reaches $B^{\frac{1}{p(1+\alpha)}}_{r,\infty}(  \mathbb{R}^n_x),$ where $\frac{1}{r}=\frac{1}{p(1+\alpha)}\left(1-\frac{2}{p}\right)+\frac{1}{p}.$ By the interpolation between $H^{\frac{1}{2(1+\alpha)}}$ and $L^p$, we have $W^{\frac{1}{p(1+\alpha)^2},r}\subset B^{\frac{1}{p(1+\alpha)^2}}_{r,r}.$ This shows when $\alpha=0,$ Theorem \ref{theorem2} implies Theorem \ref{theorem3}.
	
	In the other hand, by embedding $H^{\frac{1}{2(1+\alpha)}}\subset B^{\tilde{s}}_{r,2},$ where $\tilde{s}=\frac{1}{2(1+\alpha)}+n\left(\frac{1}{r}-\frac{1}{2}\right)$. Even with the dimension dependence, there are regions that embedding gives a better  regularity than interpolation. For example when $n=1$, $\tilde{s}\geq \frac{1}{p(1+\alpha)^2}$ when $p\leq 2+\frac{2}{\alpha}$. We compare  $\tilde{s}$ with the regularity obtained by Theorem \ref{theorem3}. In general for each fixed $n,$ $\tilde{s}\geq \frac{1}{p(1+\alpha)}$ when $p\leq \frac{2n}{n(1+\alpha)-1},$ which is compatible with $p>2$ only when $\alpha< \frac{1}{n}$. Hence Theorem \ref{theorem2} implies Theorem \ref{theorem3} when $0\leq \alpha< \frac{1}{n}$ and $2<p<\frac{2n}{n(1+\alpha)-1}$.
\end{itemize}

We now compare our result with \cite{Arsenio3} and \cite{Westdickenberg}, where mixed norm conditions in general dimensions were considered for the stationary transport equation 
\begin{equation}\label{stationary}
v\cdot\nabla_x f=g.
\end{equation} 
We shall take $\varepsilon=0$, in order to compare our theorem with results for (\ref{stationary}).
\begin{theorem}\label{theorem4}\cite{Westdickenberg}
	For $1<p<\frac{n}{n-1},$ if $f\in B^0_{p,q}(\mathbb{R}^{n}_x,L^{p_2}(\mathbb{R}^{n}_v))$ and $g\in  B^{0}_{p,q}(\mathbb{R}^{n}_x,L^{q_2}(\mathbb{R}^{n}_v))$ satisfy (\ref{stationary}), then $\rho_\phi \in B^S_{P,q}(\mathbb{R}^{n}_x),$  where $S=-n+1+\frac{1}{p_2'}\left[1+\frac{1}{q_2}-\frac{1}{p_2}\right]^{-1}$ and $P=\left[\frac{1}{p}-\frac{n-1}{n}\right]^{-1}$.
\end{theorem}
\begin{theorem}\label{theorem5}\cite{Arsenio3}
	When $\frac{4}{3}\leq p\leq 2$, if $f,g\in L^p(\mathbb{R}^{n}_x, L^2(\mathbb{R}^{n}_v))$ satisfy (\ref{stationary}), then $\rho_\phi\in W^{s,p}(\mathbb{R}^{n})$ for all $s< S$, where $S=\frac{1}{2}$ when $n=1,2$, and $S= \frac{1}{2}\left(3-\frac{4}{p}\right)+\frac{n}{4(n-1)}\left(\frac{4}{p}-2\right)$ when $n\geq 3.$ 
\end{theorem}

For the comparison with Theorem \ref{theorem4}, we take $q=2$ for an easier discussion with our $H^s$ result. And since Theorem \ref{theorem4} allows general integrabilities in $v$, let us consider $p_2=q_2'\geq 2$, which is the most favorable condition for our method. 

\begin{itemize}
	
	\item \textit{Under the assumption of Theorem \ref{theorem4} with $n=1$, $q=2$ and $p_2=q_2'\geq 2$. Both Theorem \ref{theorem2} and \ref{theorem4} reach the same regularity when $p=2.$ And Theorem \ref{theorem4} implies Theorem \ref{theorem2} when $p\neq 2$:}
	
	Here Theorem \ref{theorem4} reaches $B^{1/2}_{p,2}$, while Theorem \ref{theorem2} has $H^{1/p'}$ when $p\leq 2$ and $H^{1/2}$ when $p>2$, as mentioned in Remark \ref{remark4}. When $p=2,$ the two results are exactly the same. When $p< 2$, the spaces $B^{1/2}_{p,2}$ and $H^{1/p'}$ have the same scaling, and $B^{1/2}_{p,2}\subset H^{1/p'}$ by embedding. At last for $p>2$, $H^{1/2}\subset B^{1/2}_{p,2}$.
	
	\vspace{10 pt}
	
	Notice for $n\geq 2,$ Theorem \ref{theorem4} no longer applies to $p> 2$, same as Theorem \ref{theorem5}. The restriction $p< 2$ is not the best situation for our method, but the comparison is still interesting under these mixed norm conditions.
	
	\item \textit{Under the assumption of Theorem \ref{theorem4} with $n\geq 2$ (which forces $1<p<2$), $q=2$ and $p_2=q_2'\geq 2$, our result implies Theorem \ref{theorem4}:}
	
	In this case Theorem \ref{theorem4} gets $ B^{3/2-n}_{P,2}$ with $P=\left[\frac{1}{p}-\frac{n-1}{n}\right]^{-1}$, and our method reaches $ H^{\frac{1}{2}\left[1-\frac{2n}{p}+n\right]}$ as mentioned in Remark \ref{remark4}. Our result has more differentiability but less integrability. Moreover, by the embedding $H^{\frac{1}{2}\left[1-\frac{2n}{p}+n\right]} \subset B^{\tilde{s}}_{P,2},$ where $\tilde{s}=\frac{1}{2}\left[1-\frac{2n}{p}+n \right]+n\left(\frac{1}{P}-\frac{1}{2}\right)=\frac{3}{2}-n.$ 
	
	\item \textit{Under the assumption of Theorem \ref{theorem5}, the result by Theorem \ref{theorem2} has more integrability but less differentiability than Theorem \ref{theorem5}. Furthermore, Theorem \ref{theorem5} implies Theorem \ref{theorem2} when $n=1$ and $2,$ but the implication does not hold for $n\geq 3$:} 
	
	Under this assumption, we again have $ H^{s}$ with $s<\frac{1}{2}\left[1-\frac{2n}{p}+n\right]$. For both $n=1$ and $2,$  $W_x^{1/2,p}\subset H_x^{\frac{1}{2}\left[1-\frac{2n}{p}+n\right]}$ by Sobolev embedding. As for $n\geq 3,$ $W^{s,p}_x \subset  H_x^{\tilde{s}}$ where $s= \frac{1}{2}\left(3-\frac{4}{p}\right)+\frac{n}{4(n-1)}\left(\frac{4}{p}-2\right)$ and $\tilde{s}=\frac{1}{2}\left(3-\frac{4}{p}\right)+\frac{n}{4(n-1)}\left(\frac{4}{p}-2\right)+n\left(\frac{1}{2}-\frac{1}{p}\right).$ Notice $\tilde{s}<\frac{1}{2}\left[1-\frac{2n}{p}+n\right]$ for all $p<2$ and $n\geq 3$, so Theorem \ref{theorem5} cannot imply Theorem \ref{theorem2} in this case.
\end{itemize}

\subsection{On the non-degeneracy conditions}
%
%
The assumption (\ref{condition}) we imposed for Theorem \ref{theorem2} is different from the classical conditions on $a(v)$ in the previous literature, called the non-degeneracy condition:
\begin{definition}
	$a\in Lip(\mathbb{R}^n,\mathbb{R}^m)$ satisfies the \textbf{non-degeneracy condition of order $\nu\in (0,1]$}, if there exists $c_0>0$ such that for all compact set $D\subset\mathbb{R}^n$, 
	\begin{equation}\label{nondeg}
	\mathcal{L}^n(\left\{v\in D: |a(v)\cdot \sigma -\tau|\leq \alpha/2\right\})\leq c_0 \alpha^{\nu},
	\end{equation}
	for all $\sigma\in\mathbb{S}^{m-1}$ and $\tau\in\mathbb{R},$ where $\mathcal{L}^n$ is the Lebesgue measure in $\mathbb{R}^n$.
\end{definition}
Our assumption (\ref{condition}) is stronger than (\ref{nondeg}) with $\nu =1-\frac{1}{\gamma}.$ Indeed, when $n=m$, the assumption $J_{a^{-1}}\in L^\gamma_v$ implies (\ref{nondeg}) with $\nu=1-\frac{1}{\gamma}$, but the other direction holds only when $n=\nu=1$. When $n>1,$ (\ref{nondeg}) only gives restrictions on the pre-images of bands. And when $\nu<1,$ one can construct a Lipschitz function $a_\nu$ on $\mathbb{R}$ satisfying (\ref{nondeg}), and a sequence of measurable sets $\mathcal{O}^i$ such that $\frac{|a_\nu^{-1}(\mathcal{O}^i)|}{|\mathcal{O}^i|^\nu}\rightarrow \infty$ as $i\rightarrow \infty,$ which shows $J_{a^{-1}}\not\in L^\gamma.$ An example of construction can be found in Appendix A. 

The dimension of interests is $n\leq m$ for applications, especially when $n=1$ for scalar conservation laws. In an attempt to weaken the assumption to non-degeneracy condition with general $n\leq m$ cases, we do a different change of variables $v \mapsto \lambda=
a(v)\cdot \frac{\xi}{|\xi|}$, where $\xi$ is the frequency variable of $x,$ and our method can recover the traditional result in $L^2$ for $\nu=1.$
\begin{theorem}\label{theorem6}
	Given $ n\leq m$, $\alpha\geq 0$, $T>0$ and $0<\varepsilon\leq 1.$ Assume $a\in Lip(\mathbb{R}^n,\mathbb{R}^m)$ satisfies the non-degeneracy condition (\ref{nondeg}) with $\nu=1.$ Let $f\in L^\infty([0,T], L^{2}(\mathbb{R}^m_x \times \mathbb{R}^n_v))$ solve (\ref{transport}) for some $g\in L^1([0,T], L^{2}(\mathbb{R}^m_x \times \mathbb{R}^n_v))$, then for any $\phi \in C_c^\infty,$ one has $\rho_\phi(t,x)\in L^2\left([0,T], H^{\frac{1}{2(\alpha+1)}}(B_R(x_0)) \right),$ and
	\[
	\|\rho_\phi\|^2_{L^2\left([0,T], H^{\frac{1}{2(\alpha+1)}}(\mathbb{R}^m_x)\right)} \leq C \left(\|f\|_{L^\infty([0,T], L^{2}(\mathbb{R}^m_x\times \mathbb{R}^n_v)) }^2+\|g\|^2_{L^1([0,T], L^{2}(\mathbb{R}^m_x\times \mathbb{R}^n_v)) }\right),
	\]
	where $C$ only depends on $c_0$ and $Lip(a)$.
\end{theorem}

This $L^2$ theorem recovers the same regularity $H^{\frac{1}{2(\alpha+1)}}$ in $x$ as in \cite{DiPerna3} and \cite{DiPerna5}. Even though this regularity result is not new, we provide a different approach for proving this theorem. As we mentioned in the discussion after Corollary \ref{corollary1}, some interesting features which are also inherited by Theorem \ref{theorem3} include:
\begin{itemize}
	\item Potential applications to hydrodynamic limits as our results are independent of $\varepsilon$.
	\item The absence of Fourier transform in time variable which enables potential extensions of our method for time-discretized or stochastic kinetic equations.
\end{itemize} 

\begin{remark}
	We were unable to obtain a $L^p$ statement as we did in Theorem \ref{theorem2}. This is because the natural multiplier for the alternate proof here is not a Calderon-Zygmund operator, and we lose bounds in general $L^p$ spaces. 
	In fact, when $a(v)=v$, the corresponding multiplier would be in the form of $S\left(\frac{\xi}{|\xi|}\cdot\zeta\right),$ where $\zeta$ is the frequency variable of $v$. If $S$ is smooth, the inverse Fourier transform of this type of "directed multiplier" in two-dimension is in the form of $\frac{x\cdot v}{|x|^3} \tilde{S}\left(\frac{x^\perp\cdot v}{|x|}\right),$ which is not bounded on $L^p_{x,v}.$
\end{remark}

\begin{remark}
	We use the non-degeneracy condition as a constraint on the measures of pre-images of intervals. We extend this condition from intervals to general measurable sets, so that this is equivalent to a constraint on the determinant of Jacobian matrices and a proof similar to the one of Theorem \ref{theorem2} follows. But when $\nu <1,$  the extension from intervals to measurable sets fails (see the counterexample in the Appendix) and hence the strategy is not applicable directly here. 
\end{remark}

\section{An example of future perspective: Regularizing effects for measure-valued solutions to scalar conservation law}\label{section3}
Among several potential applications of the new method for averaging lemmas presented here, this section focuses on the regularity of so-called measure-valued solutions of conservation laws and in particular scalar conservation laws.

Scalar conservation laws can be viewed as a simplified model of hyperbolic systems which still captures some of the basic singular structure. They read 
\begin{equation}\label{conservation}
\begin{cases}
\partial_t u + \sum_{i=1}^n \partial_{x_i}A_i(u)=0,\\
u(t=0,x)=u_0(x),
\end{cases}
\end{equation}
where $u(t,x):\mathbb{R}^+\times\mathbb{R}^n\rightarrow \mathbb{R}$ is the scalar unknown and $A: \mathbb{R} \rightarrow \mathbb{R}^n$ is a given flux.

The concept of measure-valued solutions to hyperbolic systems such as \eqref{conservation} had already been introduced in \cite{DiPerna1}. It has recently seen a significant revival of interest as measure-valued solutions offer a more statistical description of the dynamics, see in particular \cite{FKMT,Fjordholm}. 

It is convenient to define measure-valued solution through the kinetic formulation of \eqref{conservation}, which also allows for a straightforward application of our results. A scalar function $u(t,x)\in L^\infty(\R_+,\ L^1(\R^n)$ corresponds to a measure-valued solution if there exists $f(t,x,v)\in L^\infty(\R_+\times\R^n\times \R)$ with the constraint
\begin{equation}
u(t,x)=\int_\R f(t,x,v)\,dv,\quad -1\leq f\leq 1,\label{measureval}
\end{equation}
and if $f$ solves the kinetic equation
\begin{equation}
\partial_t f+a(v)\cdot\nabla_x f=\partial_v m,
\label{kinetic}
\end{equation}
for $a(v)=A'(v)$ and any finite Radon measure $m$. If $u$ is obtained as a weak-limit of a sequence $u_n$ then $f$ includes some information on the oscillations of $u_n$ since it can directly be obtained from the Young measure $\mu$ of the sequence
\[
f(t,x,v)=\int_0^v \mu(t,x,dz).
\]
The system \eqref{measureval}-\eqref{kinetic} is hence immediately connected to the notion of kinetic formulation for scalar conservation laws introduced in the seminal article \cite{Lions2} and extended to isentropic gas dynamics in \cite{Lions3}. If $u$ is an entropy solution to \eqref{conservation}, then one may define
\begin{equation}
f(t,x,v)=\left\{\begin{aligned}
& 1\quad\mbox{if}\ 0\leq v\leq u(t,x),\\
& -1\quad\mbox{if}\  u(t,x)\leq v <0,\\
& 0\quad \mbox{otherwise},
\end{aligned}\right.
\label{kineticformulation}
\end{equation}
and $f$ solves the kinetic equation \eqref{kinetic} with the additional constraint that $m\geq 0$ which corresponds to the entropy inequality.

We refer for example to \cite{Perthame1} for a thorough discussion of kinetic formulations and their usefulness, such as recovering the uniqueness of the entropy solution first obtained in \cite{Kr}.

The use of kinetic formulations has proved effective in particular in obtaining regularizing effects for scalar conservation laws. In one dimension and for strictly convex flux,  Oleinik \cite{Oleinik} proved early that entropy solutions are regularized in $BV$. In more than one dimension and for more complex flux that are still non-linear in the sense of \eqref{nondeg} with $\nu=1$, a first regularizing effect had been obtained in \cite{Lions2} yielding $u\in W^{s,p}$ for all $s<1/3$ and some $p>1$.

Such regularizing effects actually do not use the sign of $m$ and for this reason hold for any weak solution to \eqref{conservation} with bounded entropy production. Among that wider class a counterexample constructed in \cite{DeLellis} proves that solutions cannot in general be expected to have more than $1/3$ derivative. The optimal space $(B_{1/3,3}^\infty)_{x,loc}$ was eventually derived in  \cite{Golse4}. Whether a higher regularity actually holds for entropy solutions (instead of only bounded entropy production) remains a major open problem though.

It had been observed in \cite{Jabin2} that the regularizing effect for the kinetic formulation relies in part in the regularity of the function $f$ defined by \eqref{kineticformulation}: For example such an $f$ belongs to $L^\infty(\R_+\times \R^n,\ BV(\R))$. Unfortunately such additional regularity is lost for measure-valued solutions since we only have $f\in L^1\cap L^\infty$ by \eqref{measureval}.

A priori, one may hence only apply the standard averaging result from \cite{DiPerna5} directly on \eqref{kinetic}. Assuming non-degeneracy of the flux, {\em i.e.} \eqref{nondeg} with $\nu=1$, we may apply Theorem~\ref{DiPerna-Lions}  for any $\alpha>1$, $g\in L^1$ and $f\in L^2$ (the optimal space for this theorem). One then deduces that if $u$ corresponds to a measure-valued solution with $f$ compactly supported in $v$ then $u\in B^s_{5/3,2}$ for any $s<1/5$.

However we are then making no use of the additional integrability of $f$. Instead one may also apply our new result Theorem~\ref{maintheorem} to (\ref{kinetic}) with
\begin{corollary}
	Let $f$ satisfy \eqref{measureval} and solve \eqref{kinetic} for some finite Radon measure $m$ and some $a: \R^n\to\R^n$ with \eqref{condition} for $\gamma=\infty$. Assume moreover that $f\in L^\infty([0,\ T],\ L^1(\R^n\times\R^n))$ and is compactly supported in velocity. Then $u\in L^2([0,\ T],\ H^s(\R^n))$ for any $s<1/4$.
	\label{cor:measureval}
\end{corollary}
In dimension $1$, Corollary \ref{cor:measureval} directly applies to measure-valued solutions and improve the regularity from almost $B^{1/5}_{5/3,2}$ in $x$
to almost $H^{1/4}$. In higher dimensions, as we observed, we cannot directly replace \eqref{condition} with \eqref{nondeg}. Therefore a better understanding of the regularity of measure-valued solutions is directly connected to further investigations of what should replace \eqref{condition} if $a:\R^m\to\R^n$ with $m<n$.  
\section{Proofs }\label{section4}
\subsection{Proof of Theorem \ref{theorem2}}
\subsubsection{Main proof}
The proof contains mainly three steps as follows.

\textit{ \textbf{Step 1:} Preparations: localization, regularization and change of variables.}

As the result is local, we assume $f$ is compactly supported in $x$ for convenience. Fix a compactly supported function $\phi(v) \in W^{\alpha,\infty}_v$. Without loss of generality, assume $supp(\phi)\subseteq B(0,1).$ Consider $f \phi,$ which satisfies
$$
\varepsilon \partial_t(f\phi)+a(v) \cdot \nabla_x (f\phi)=(-\Delta_v)^{\alpha/2}g\phi.
$$ 
We denote the Fourier transform of $f$ in $x$ by $\tilde{f}.$ Fix a smooth function $\Phi(v)$ with $supp(\Phi)\subseteq B(0,1)$. Consider $F_{s_1}=(\tilde{f}\phi)\star_v \Phi_{|\xi|^{-s_1}}$, where 
$\Phi_{|\xi|^{-s_1}}(v)=|\xi|^{ns_1} \Phi(v|\xi|^{s_1})$ with $s_1\geq 0$ to be decided later.  Notice
$$
supp(F_{s_1})\subseteq \overline{supp(\phi)+ supp(\Phi_{|\xi|^{-s_1}})}\subseteq \overline{ B(0,1+|\xi|^{-s_1})}\subset \overline{B(0,2)}
$$
is of compact support for all $|\xi|\geq 1$. And it satisfies
\begin{equation}\label{E:reg}
\varepsilon \partial_t F_{s_1} + ia(v) \cdot \xi F_{s_1}=
((-\Delta_v)^{\alpha/2}\tilde{g}\phi)\star_v \Phi_{|\xi|^{-s_1}}+Com^1,
\end{equation}
where $F_{s_1}=(\tilde{f}\phi)\star_v \Phi_{|\xi|^{-s_1}}$ and the commutator term
$$
Com^1(v)=i\int (a(v)-a(w))\cdot \xi \tilde{f}(w)\phi(w)\Phi_{|\xi|^{-s_1}}(v-w)\, dw.
$$

Note the usage of localization in $v$ will be more clear in the last step of our proof.

By change of variables $v\mapsto v'=a(v)$, (\ref{E:reg}) can be  rewritten as 
\begin{equation}\label{E:h}
\varepsilon \partial_t h +i v' \cdot \xi  h = k^1 +k^2 
\end{equation}
in the sense of distribution, where $h, k^1$ and $k^2$ are defined as follows:
$$
\int h (v') \psi(v')\, dv'=\int F_{s_1} (v) \psi(a(v)) \, dv .
$$
$$
\int k^1(v')\psi(v')\, dv'=\int \left[((-\Delta_v)^{\alpha/2}\tilde{g}\phi)\star_v \Phi_{|\xi|^{-s_1}} \right](v)\psi(a(v))\, dv,
$$
and
$$
\int k^2(v')\psi(v')\, dv'=\int Com^1(v) \psi(a(v)) \, dv.
$$
\textit{\textbf{Step 2:} Commutator method with $m_0$ on $h.$}
Consider a smooth radial bump function $\chi(\xi)$ with support on $\frac{1}{2}<|\xi|<2,$ such that $\sum_{k\in\mathbb{Z}}\chi(2^{-k}\xi)\equiv 1,$ for all $\xi\neq 0.$ For each $k\in\mathbb{N},$ we apply commutator method with $m_0$ on $h(v')\chi\left(2^{-k}\xi\right)$ and get
\begin{equation}\label{comm_h}
\begin{split}
&\,\,\,\,\,\,\,\int |\mathcal{F}_{\zeta'}(h)|^2(\zeta') \frac{2^k}{(1+|\zeta'|^2)^{3/2}}  \, d\zeta' \, dt \, \chi(2^{-k}\xi) d\xi\\
&\lesssim \int \xi\cdot\nabla_{\zeta'}m_0(\xi,\zeta')|\mathcal{F}_{\zeta'}(h)|^2 \, \chi(2^{-k}\xi) d\xi \, d\zeta' \, dt\\
&= \int \bar{h}(v')(\frac{1}{i}\frac{\xi}{|\xi|}\cdot\nabla_{v'} G_1^n \star_{v'} h) \, dv' \chi(2^{-k}\xi)\, d\xi|_{t=0}^{t=T}\\
& \,\,\,\,\,\,\,\,\,\,\,\,  +Re\int \bar{h}(v') \frac{\xi}{|\xi|}\cdot\nabla_{v'} G_1^n \star_{v'} \left[(k^1+k^2) \right] \, dv' \chi(2^{-k}\xi)\, d\xi \, dt\\
&:=A_k 
\end{split}
\end{equation}
We estimate $A_k$ and get 
\begin{lemma}\label{lemma1}
	Denote $\mathcal{F}^{-1}_x(\chi(2^{-k}\xi) \tilde{f}\phi) $ by $f_k$ and $\mathcal{F}^{-1}_x(\chi(2^{-k}\xi) \tilde{g}\phi) $ by $g_k$. Let $p_1,p_2,q_1,q_2\in (1,\infty].$ Then for each fixed $k\in\mathbb{N}$ ,
	\begin{equation}\label{Ak}
	\begin{split}
	|A_k|&\lesssim 2^{kd_4 +ks_1 d_2 }\|f_k\|^2_{ L^{p_1}_{x}L^{p_2}_{v}}|_{t=0}^{t=T}\\
	&+2^{kd_3+ks_1 d_1 +k\alpha s_1} \int \|f_k\|_{L^{p_1}_{x}L^{p_2}_{v}}\|g_k\|_{L^{q_1}_{x}L^{q_2}_{v}} \, dt\\
	&+2^{kd_4+ks_1 d_2+k(1-s_1)}\int \|f_k\|_{L^{p_1}_{x}L^{p_2}_{v}}^2\, dt.
	\end{split}
	\end{equation}
	where $d_1=\max{ \left\{n\left(\frac{1}{p_2}+\frac{1}{q_2}-\ell\right) ,0  \right\}   }$, $d_2=\max{ \left\{n\left(\frac{2}{p_2}-\ell\right) ,0  \right\}   }$, with $\ell=\frac{\gamma-2}{\gamma-1}$,  and $d_3=\max{ \left\{n\left(\frac{1}{p_1}+\frac{1}{q_1}-1\right) ,0  \right\}   }$, $d_4=\max{ \left\{n\left(\frac{2}{p_1}-1\right) ,0  \right\}   }$.
\end{lemma}
To minimize the order of $\xi$ of the sum in (\ref{Ak}), we choose 
$$
s_1=\min\left\{\frac{1-(d_3-d_4)}{\alpha+1 +(d_1-d_2)},1 \right\},
$$
and so the highest order is $1-S$, where $S=s_1( 1-d_2)-d_4$.

Divide the whole inequality (\ref{comm_h}) with $2^{k(1-S+\delta) }$ for any small $\delta>0,$ then we attain
\begin{equation}\label{bound_k}
\begin{split}
&\int |\mathcal{F}_{\zeta'}(h)|^2 \frac{|\xi|^{(S-\delta)}}{(1+|\zeta'|^2)^{3/2}}  \, d\zeta' \, dt \, \chi(2^{-k}\xi) d\xi\\
&\lesssim 2^{-k\delta }\left[\|f_k\|^2_{ L^{p_1}_{x}L^{p_2}_{v}}|_{t=0}^{t=T}+ \int \|f_k\|_{L^{p_1}_{x}L^{p_2}_{v}}\|g_k\|_{L^{q_1}_{x}L^{q_2}_{v}} \, dt+ \int \|f_k\|_{L^{p_1}_{x}L^{p_2}_{v}}^2 \, dt\right],
\end{split}
\end{equation} 
for all $k\in\mathbb{N}.$ 

The same inequality can be obtained even if any of $p_1,p_2,q_1,q_2$ is equal to $1$, because the additional logarithm appears from the weak boundedness of Calderon-Zygmund operator would not affect the  argument.

Sum over $k\in \mathbb{N}$ for (\ref{bound_k}), we get 
\[
\begin{split}
&\int \chi_0(\xi)|\xi|^{s} \bar{h }(v)G^n_{3}(v-w)h (w)\, dw \, dv \, dt \, d\xi \\
&\qquad\qquad
\lesssim \|f\|^2_{L^\infty\left([0,T], L^{p_1}_{x}L^{p_2}_{v}\right)}+\|g\|^2_{L^1\left([0,T], L^{q_1}_{x}L^{q_2}_{v}\right)},
\end{split}
\]
with $s<S=(1-d_2)\min\left\{\frac{1-(d_3-d_4)}{\alpha+1 +(d_1-d_2)},1 \right\}-d_4$, and $\chi_0(\xi):=\sum_{k\in\mathbb{N}}\chi(2^{-k}\xi)$. 

The last step is to translate the quadratic form of $h$ back to a norm of velocity average of $f.$

\textit{\textbf{Step 3:} Derive result back to $f.$}

With the change of variables again we have
\[\begin{split}
&\int \int \left|\int  F_{s_1} (v)\psi(a(v)) \, dv \right|^2 |\xi|^{s} \, d\xi \, dt \\
&\qquad\qquad=\int \int \left|\int  h (v')\psi(v') \, dv' \right|^2 |\xi|^{s} \, d\xi \, dt <\infty,
\end{split}
\]
for all $\psi\in H^{3/2}$. By the assumptions that $\phi$ and $\Phi$ are compactly supported in $v$, one can show
\begin{lemma}\label{lemma2}
	There exists $\psi\in H^{3/2}$ such that 
	\[
	\int_0^T \int_{|\xi|\geq 1} \left|\int  \tilde{f}\phi \, dv \right|^2 |\xi|^{s} \, d\xi\, dt\lesssim \int \int \left|\int  F_{s_1} (v)\psi(a(v)) \, dv \right|^2 |\xi|^{s} \, d\xi \, dt<\infty
	\]
	for all $s<S.$
\end{lemma}
This concludes our proof.

\begin{remark}
	Note that $m(\xi,\zeta)$ to be homogeneous zero in $\zeta$ is essential for the commutator to be positive-definite after interacting with the transport operator. In fact, if consider 
	$m(\xi,\zeta)=\frac{\xi}{|\xi|}\cdot \frac{\zeta}{(1+|\zeta|^2)^{\beta/2}}$ with $\beta>1,$ 
	$$
	\xi\cdot\nabla_\zeta m =\frac{|\xi|\left[ (1+|\zeta|^2)-\beta |\frac{\xi}{|\xi|}\cdot\zeta|^2 \right]}{(1+|\zeta|^2)^{\beta/2+1}}.
	$$
	When $\zeta$ is parallel to $\xi$ and $|\zeta|$ is large, it is negative and the argument doesn't work.
	
	The regularization recollects the regularization process in \cite{DiPerna2}. Here the convolution with $\Phi_{|\xi|^{-s_1}}$, along with the multiplier $m_0$, show explicitly the interaction between the regularity in $x$ and $v$.
\end{remark}
%
\subsubsection{Proof of Lemma \ref{lemma1}}
%
Before estimating $A_k$, let us first show the relation of functions connected through change of variables.

\begin{proposition}\label{proposition1}
	Let $a\in Lip(\mathbb{R}^n).$ If $J_{a^{-1}}\in L^\gamma$, the change of variables is bounded from $L^p$ to $L^{(p'\gamma')'}$. Precisely, if $\int \ell(v')\psi(v')\, dv' = \int L(v)\psi(a(v)) \, dv, $ then
	$$ 
	\|\ell\|_{L^{(p'\gamma')'}_{v'}}\lesssim \|L\|_{L^p_v}.
	$$
\end{proposition}

\begin{proof}
	By H\"{o}lder's inequality,
	$$
	\int |\psi(a(v))|^{p'} \, dv =\int |\psi(v')|^{p'} J_{a^{-1}}(v') \, dv' \leq \|J_{a^{-1}}\|_{L^\gamma}\left(\int |\psi(v')|^{p'\gamma'}\, dv'\right)^{1/\gamma'}.
	$$
	So
	\begin{align*}
	\|\ell\|_{L^{(p'\gamma')'}_{v'}}
	=\sup_{\|\psi\|_{L^{p'\gamma'}_{v'}}=1}\left|\int \ell \psi \right|
	&=\sup_{\|\psi\|_{L^{p'\gamma'}_{v'}}=1}\left|\int L(v)\psi(a(v)) \right|\\
	&\leq \sup_{\|\psi\|_{L^{p'\gamma'}_{v'}}=1} \|L\|_{L^p_v}\|\psi(a(v))\|_{L^{p'}_v}\lesssim \|L\|_{L^p_v}.
	\end{align*}
\end{proof}

\begin{remark}
	If $a$ is one-to-one and $a\in Lip(\mathbb{R}^n,\mathbb{R}^m)$, where $n< m,$ the area formula gives
	\[
	\int |\psi(a(v))|^{p'} \, dv =\int |\psi(v')|^{p'} |J_{a^{-1}}(v')| \, d\mathcal{H}^n(v'),
	\]
	with $|J_{a^{-1}}|=(det(Da^{-1})(Da^{-1})^T)^{1/2}$, and $\mathcal{H}^n$ is a Hausdorff measure of dimension $n$. This relation would put $\ell$ in Hausdorff measurable spaces, which are not compatible with our arguments with Fourier analysis in the whole space.
\end{remark}

We now use Proposition \ref{proposition1} to estimate $A_k$ term by term when $d_i>0$ for all $i=1,2,3,4$. The other cases follow similar calculations.
\begin{itemize}
	\item 
	For the first term: By the Cauchy-Schwarz inequality, and that $R\cdot\nabla_v G^n_1$ is Calderon-Zygmund operator:
	\begin{equation}\label{first_term}
	\begin{split}
	&  \int\int \bar{h}\left(\frac{\xi}{|\xi|}\cdot\nabla_{v'} G_1^n \star_{v'} h\right) \, dv' \, \chi(2^{-k} \xi ) d\xi |_{t=0}^{t=T}\\
	&\qquad\qquad\leq \|\mathcal{F}^{-1}_x(h \chi(2^{-k} \xi))\|^2_{L^2_{xv'}}  |_{t=0}^{t=T},
	\end{split}
	\end{equation}

	Denote $\mathcal{F}^{-1}_x(\chi)$ by $S$. By Proposition \ref{proposition1}, for each fixed $t,$
	\begin{align*}
	\|\mathcal{F}^{-1}_x(h \chi(2^{-k} \xi))\|_{L^2_{xv'}}
	&\lesssim   \|S_{2^{-k }}\star_x f_k \star_v \Phi_{2^{-k s_1}}\|_{L^2_x L^{\frac{2(\gamma-1)}{(\gamma-2)}}_{v}}\\
	&\lesssim 2^{kn\left(\frac{1}{p_1}-\frac{1}{2}\right) +ks_1\left(\frac{1}{p_2}-\frac{\gamma-2}{2\gamma-2}\right) }\|f_k\|_{L^{p_1}_{x}L^{p_2}_{v}}.
	\end{align*}
	
	Plug this back into (\ref{first_term}) and we have
	\begin{equation*}\begin{split}
	&\int\int \bar{h}\left(\frac{\xi}{|\xi|}\cdot\nabla_{v'} G_1^n \star_{v'} h\right) \, dv' \, \chi(2^{-k} \xi ) d\xi |_{t=0}^{t=T}\\
	&\qquad\qquad\lesssim 2^{kn\left(\frac{2}{p_1}-1\right) +ks_1\left(\frac{2}{p_2}-\frac{\gamma-2}{\gamma-1}\right) }\|f_k\|^2_{L^{p_1}_{x}L^{p_2}_{v}}|_{t=0}^{t=T}.
	\end{split}
	\end{equation*}
	
	\item For the second term:
	\begin{align*}
	&\int \bar{h}\left(\frac{\xi}{|\xi|}\cdot \nabla_{v'} G^n_1 \star_{v'} k^1 \right) \, dv'\,\chi(2^{-k} \xi) d\xi \, dt\\
	& \lesssim \int \|S_{2^{-k }}\star_x f_k \star_v \Phi_{2^{-k s_1}}\|_{L^2_x L^{\frac{2(\gamma-1)}{(\gamma-2)}}_{v}}\\
	&\qquad\|2^{k\alpha s_1} S_{2^{-k}} \star_x g_k \star_v ((-\Delta_v)^{\beta}\Phi)_{2^{-ks_1}}\|_{L^2_x L^{\frac{2(\gamma-1)}{(\gamma-2)}}_{v}} \, dt,
	\end{align*}
	which is of the order of
	\[
	2^{kn\left(\frac{1}{p_1}+\frac{1}{q_1}-1\right)+ks_1 n\left(\frac{1}{p_2}+\frac{1}{q_2}-\frac{\gamma-2}{\gamma-1}\right) +k\alpha s_1} \int \|f_k\|_{L^{p_1}_{x}L^{p_2}_{v}}\|g_k\|_{L^{q_1}_{x}L^{q_2}_{v}} \, dt.
	\]

	\item For the last term:
	\begin{align*}
	&\int \bar{h}\left(\frac{\xi}{|\xi|}\cdot \nabla_{v'} G^n_1 \star_{v'} k^2  \right) \, dv'\,\chi(2^{-k} \xi) d\xi \, dt\\
	&\lesssim \int \|S_{2^{-k }}\star_x f_k \star_v \Phi_{2^{-k s_1}}\|_{L^2_x L^{\frac{2(\gamma-1)}{(\gamma-2)}}_{v}}\\
	&\qquad\qquad\| S_{2^{-k}} \star_x Com^1 \star_v \Phi_{2^{-ks_1}}\|_{L^2_x L^{\frac{2(\gamma-1)}{(\gamma-2)}}_{v}} \, dt.
	\end{align*}
	
	Because $\Phi$ is compactly supported, $\Phi_{2^{-ks_1}}(v-w)$ forces $|v-w|\lesssim 2^{-ks_1}$. Moreover since $a$ is Lipschitz, $|a(v)-a(w)|\lesssim 2^{-ks_1}.$
	\begin{align*}
	&\| S_{2^{-k}} \star_x Com^1 \star_v \Phi_{2^{-ks_1}}\|_{L^2_x L^{\frac{2(\gamma-1)}{(\gamma-2)}}_{v}}\\
	&=\Bigg\|\int 2^{k}(a(v)-a(w))\cdot  (f\star_x (\nabla_x S)_{2^{-k}})(w)\phi(w)\\
	&\qquad\qquad\Phi_{2^{-ks_1}} (v-w) \, dw \Bigg\|_{L^2_{x}L^{\frac{2(\gamma-1)}{(\gamma-2)}}_v}\\
	&\lesssim 2^{k-ks_1} \|   f\star_x (\nabla_x S)_{2^{-k}}|\star_v \Phi_{2^{-ks_1}}   \|_{L^2_{x}L^{\frac{2(\gamma-1)}{(\gamma-2)}}_v}\\
	&\lesssim 2^{kns_1\left(\frac{1}{p_2}-\frac{(\gamma-2)}{2(\gamma-1)}\right)+kn\left(\frac{1}{p_1}-\frac{1}{2}\right)+k(1-s_1)}\|f\|_{L^{p_1}_{x}L^{p_2}_{v}}.
	\end{align*}
	
	Hence
	\begin{align*}
	&\int \bar{h}\left(\frac{\xi}{|\xi|}\cdot \nabla_{v'} G^n_1 \star_{v'} k^3  \right) \, dv'\,\chi(2^{-k} \xi) d\xi \, dt \\
	&\lesssim 2^{kn\left(\frac{2}{p_1}-1\right)+ks_1 n\left(\frac{2}{p_2}-\frac{\gamma-2}{\gamma-1}\right)+k(1-s_1)}\int \|f\|_{L^{p_1}_{x}L^{p_2}_{v}}^2\, dt.
	\end{align*}

\end{itemize}

Combining all estimates,
\begin{align*}
|A_k|&\lesssim 2^{kn\left(\frac{2}{p_1}-1\right) +ks_1 n  \left(\frac{2}{p_2}-\frac{\gamma-2}{\gamma-1}\right) }\|f_k\|^2_{ L^{p_1}_{x}L^{p_2}_{v}}|_{t=0}^{t=T}\\
&+2^{kn\left(\frac{1}{p_1}+\frac{1}{q_1}-1\right)+ks_1 n\left(\frac{1}{p_2}+\frac{1}{q_2}-\frac{\gamma-2}{\gamma-1}\right) +k\alpha s_1} \int \|f_k\|_{L^{p_1}_{x}L^{p_2}_{v}}\|g_k\|_{L^{q_1}_{x}L^{q_2}_{v}} \, dt\\
&+2^{kn\left(\frac{2}{p_1}-1\right)+ks_1 n\left(\frac{2}{p_2}-\frac{\gamma-2}{\gamma-1}\right)+k(1-s_1)}\int \|f_k\|_{L^{p_1}_{x}L^{p_2}_{v}}^2\, dt.
\end{align*}
\subsubsection{Proof of Lemma \ref{lemma2}}
Choose two smooth functions $\psi_1$ and $\psi_2$ such that $\psi_1(a(v))\equiv 1$ on $v\in B(0,1),$ and $\psi_2(v)\equiv 1$ on $v\in B(0,2).$ We put $\psi_1$ in the place of $\psi$ and plug in $\psi_2$ as an auxiliary function at no cost since it's $1$ on the support of $\phi$. Then
\begin{align*}
\infty &> \int \int_{|\xi|\geq 1} \left|\int  F_{s_1} (v) \psi_2(v) \, dv \right|^2 |\xi|^{s} \, d\xi \, dt\\
&=\int \int_{|\xi|\geq 1} \left|\int  \mathcal{F}_\zeta(\tilde{f}\phi)(\zeta)\mathcal{F}_{\zeta}(\Phi)(\zeta |\xi|^{-s_1}) \mathcal{F}_{\zeta}(\psi_2)(\zeta) \, d\zeta \right|^2 |\xi|^{s} \, d\xi \, dt\\
&= \int \int_{|\xi|\geq 1} \left|\int  (\tilde{f} \phi)(\Phi_{|\xi|^{-s_1}} \star_v \psi_2) \, dv \right|^2 |\xi|^{s} \, d\xi \, dt.
\end{align*}

Because $\psi_2\equiv 1$ on $B(0,2)$ and $|v-w|\leq |v|+|w|\leq 1+|\xi|^{-s_1}\leq 2$ when $|\xi|\geq 1,$
\begin{align*}
(\Phi_{|\xi|^{-s_1}} \star_v \psi_2 )(v)&=\int |\xi|^{ns_1}\Phi(w|\xi|^{s_1})\psi_2(v-w) \, dw\\
&=\int |\xi|^{ns_1}\Phi(w|\xi|^{s_1}) \, dw= \|\Phi\|_{L^1_v}  \,\,\,\,\,\,\text{for all $|v|\leq 1.$}
\end{align*}

So finally we reach
$$
\int_0^T \int_{|\xi|\geq 1} \left|\int  \tilde{f}\phi \, dv \right|^2 |\xi|^{s} \, d\xi\, dt<\infty
$$
for all $s<S.$
%
\subsection{Proof of Theorem \ref{theorem6}}
%
This proof is essentially the same as Theorem \ref{theorem2}, but with a different change of variable. After Step 1, instead of $v \mapsto v'=a(v)$ , we make $v \mapsto \lambda=a(v)\cdot \frac{\xi}{|\xi|}$ for each fixed $\xi$. For convenience, let us denote $\epsilon=|\xi|^{-s_1}$. So parallel to (\ref{E:h}), we have
\begin{equation} 
\partial_t h_\epsilon +i \lambda |\xi|  h_\epsilon = k^1_\epsilon +k^2_\epsilon 
\end{equation}
in the sense of distribution, where $h_\epsilon, k^1_\epsilon,k^2_\epsilon, k^3_\epsilon$ are defined as following:
$$
\int F_\epsilon (v) \psi\left(a(v)\cdot\frac{\xi}{|\xi|}\right) \, dv =\int h^\xi_\epsilon (\lambda) \psi(\lambda)\, d\lambda.
$$
$$
\int k^1_\epsilon(\lambda)\psi(\lambda)\, d\lambda=\int \left[((-\Delta_v)^{\alpha/2}\tilde{g}\phi)\star_v \Phi_\epsilon \right](v)\psi\left(a(v)\cdot \frac{\xi}{|\xi|}\right)\, dv,
$$
and
$$
\int k^2_\epsilon(\lambda)\psi(\lambda)\, d\lambda=\int Com^1(v) \psi\left(a(v)\cdot\frac{\xi}{|\xi|}\right) \, dv.
$$
The subscript $\epsilon$ is to emphasize the dependence on $\xi$.

Thanks to the non-degeneracy condition with $\nu=1,$ this change of variables preserves $L^p$ norm:
\begin{proposition}\label{proposition2}
	Let $a$ be Lipschitz and satisfy (\ref{nondeg}) with $\nu=1.$ Let $\psi:\mathbb{R}\rightarrow \mathbb{R}.$ Then for all $\sigma \in \mathbb{S}^{m},$ $1\leq p\leq \infty,$
	$$
	\|\psi(a(v)\cdot \sigma)\|_{L^p_v}\leq c_0\|\psi\|_{L^{p}_\lambda}.
	$$
	And hence if $\int L(v) \psi \left(a(v)\cdot \sigma\right)\, dv=\int \ell^\sigma(\lambda)\psi(\lambda)\, d\lambda,$
	then $$
	\|{\ell^\sigma}\|_{L^p_{\lambda}}\lesssim \|L\|_{L^p_v}.
	$$ 
\end{proposition}

Consider
$$
\int \bar{h^\xi_\epsilon}(\lambda)\frac{1}{i}(\partial_\lambda G^1_1)(\lambda-\alpha)h^\xi_\epsilon(\alpha)\, d\alpha  \, d \lambda.
$$
Then similar estimations and procedures lead to
$$
\int |\xi|^{1/(\alpha+1)} \bar{h_\epsilon^\xi}(\lambda)G^1_3(\lambda-\alpha)h^\xi_\epsilon (\alpha)\, d\alpha\, d\lambda \, dt \, d\xi<\infty.
$$
One can conclude the result from here by following Step 3 in the proof of Theorem \ref{theorem2}.

Notice here everything is in one dimension for each fixed $\xi.$ And because of the $L^2$ setting, it is valid to do calculation in the level of $(v,\xi).$

The last thing to check is Proposition \ref{proposition2}.

\vspace{5pt}

\noindent \textbf{Proof of Proposition \ref{proposition2}.} When $p=\infty$, the result is straightforward. For $1\leq p<\infty.$ 
(\ref{nondeg}) implies for any interval $I,$ we have 
$$
m(\left\{v\in B(0,1): a(v)\cdot \sigma \in I \right\})\leq c_0 m(I).
$$

By a standard approximation from intervals to general measurable sets, one has for any measurable set $A$,
$$
m(\left\{v\in B(0,1): a(v)\cdot \sigma \in A \right\})
\leq c_0 m(A).
$$

From this we see the relation between the distribution functions of $\psi(a(v)\cdot \sigma)$ and $\psi$: 
\begin{align*}
d_{\psi(a(v)\cdot \sigma)}(s) &=m(  \left\{v\in B: a(v)\cdot \sigma \in \left\{\lambda : |\psi(\lambda)|>s \right\}  \right\}  )\\
&\leq c_0 m(  \left\{\lambda : |\psi(\lambda)|>s \right\}  ) =c_0  d_{\psi}(s).
\end{align*}

Therefore
\begin{equation}\label{distribution}
\begin{split}
\|\psi(a(v)\cdot \sigma)\|_{L^p_v} 
&=p^{1/p}\left(\int^\infty_0 \left[ d_{\psi(a(v)\cdot \sigma)}(s)^{1/p} s  \right]^p \frac{ds}{s} \right)^{1/p}\\
&\leq p^{1/p}\left(\int^\infty_0 \left[ c_0^{1/p} d_{\psi}(s)^{1 /p} s  \right]^p \frac{ds}{s} \right)^{1/p}= c_0   \|\psi\|_{L^{p}_\lambda}.
\end{split}
\end{equation}
And by duality,
\begin{align*}
\|\ell^\sigma\|_{L^{p}_{\lambda}}
=\sup_{\|\psi\|_{L^{p'}_{\lambda}}=1}\left|\int \ell^\sigma \psi \right|
&=\sup_{\|\psi\|_{L^{p'}_{\lambda}}=1}\left|\int L(v)\psi\left(a(v)\cdot\sigma\right) \right|\\
&\leq \sup_{\|\psi\|_{L^{p'}_{\lambda}}=1} \|L\|_{L^p_v}\left|\left|\psi\left(a(v)\cdot\sigma\right)\right|\right|_{L^{p'}_v}\\
& \leq c_0 \sup_{\|\psi\|_{L^{p'}_{\lambda}}=1} \|L\|_{L^p_v}\|\psi\|_{L^{p'}_\lambda}=c_0 \|L\|_{L^p_v}, 
\end{align*}
where the first inequality is due to the H\"{o}lder's inequality, and second by (\ref{distribution}). This concludes our proof for Proposition \ref{proposition2} and hence Theorem \ref{theorem6}.

\section*{Acknowledgement}
The authors would like to thank M. Machedon for sharing his insight on our method.

$$
\,
$$
$$
\,
$$
$$
\,
$$

\section*{Appendix: Example for the non-degeneracy condition}

We say $a(v)\in Lip(\mathbb{R})$ satisfies (\ref{nondeg}) with $\nu\in(0,1]$ on intervals if
\begin{equation}\label{I}
|\left\{v: a(v) \in I \right\}|\leq C|I|^\nu, \,\,\,\,\,\,\text{for all intervals $I,$}
\end{equation}
And $a(v)$ satisfies the non-degeneracy condition on open sets with $\nu\in (0,1]$:
\begin{equation}\label{O}
|\left\{v: a(v) \in \mathcal{O} \right\}|\leq C|\mathcal{O}|^\nu, \,\,\,\,\,\,\text{for all open set $\mathcal{O}.$}
\end{equation}

Here we give an example to show (\ref{I}) cannot imply (\ref{O}) with the same $\nu$ when $\nu=1/2.$ In fact the construction can be adapted to produce examples for all $\nu<1.$ Notice (\ref{I}) and (\ref{O}) are equivalent when $\nu=1$.

Define $a:\left[0,\sum_{i=0}^\infty \frac{1}{3^i}\right] \rightarrow \left[0,\sum_{i=0}^\infty \frac{1}{3^{2i}}\right]\subset \mathbb{R} $ as follows:

\[
\text{on $[0,1]=D_1,$}  \,\,\,\,\,\,  a(v)=a_1(v)=1-(1-v)^2, 
\] 
\[
\text{on $\left[1,1+\frac{1}{3}\right]=D_2,$} \,\,\,\,\,\,  a(v)= a_2(v)=1+\frac{1}{3^2}a_1((v-1)3)
\] 
\[
\vdots
\]

\begin{figure}[h]
\begin{center}
\scalebox{0.9}{\includegraphics{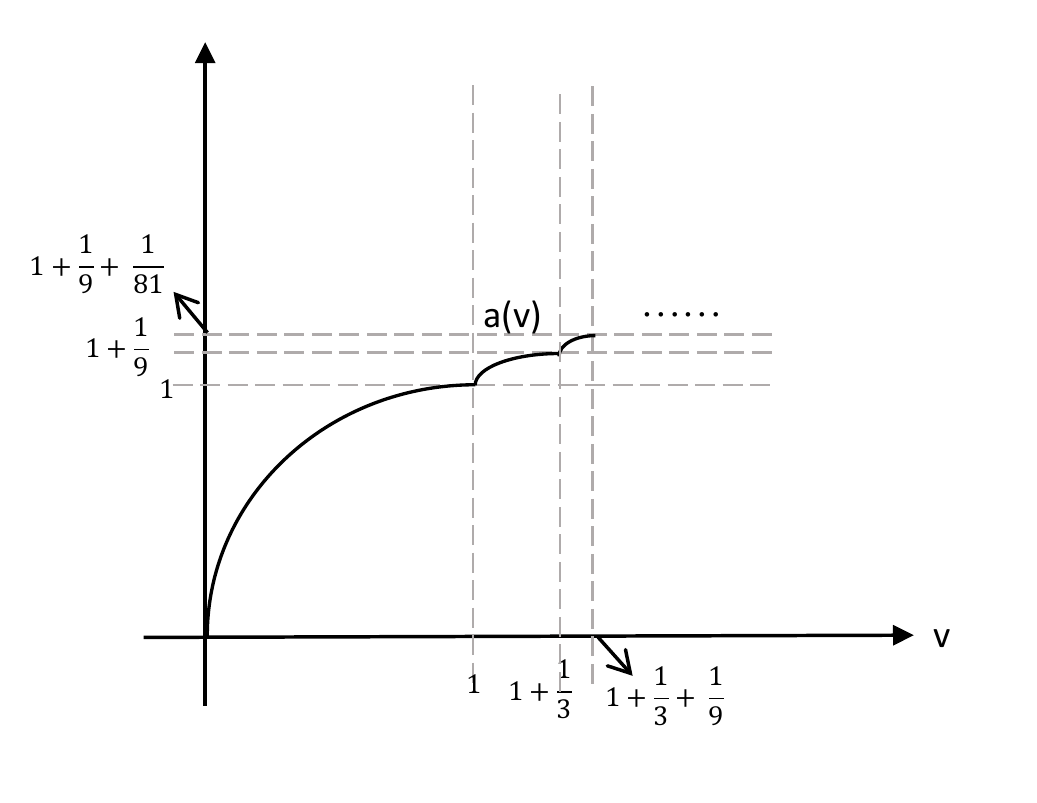}} \caption{graph of $a(v)$}
\end{center}
\end{figure}

The general formula is
\[
a(v)=a_n(v)=\sum_{i=0}^{n-2} \frac{1}{3^{2i}}+\frac{1}{3^{2(n-1)}} a_1\left(\left(v-\sum_{i=0}^{n-2}\frac{1}{3^i}\right)3^{n-1}\right)\,\,\,\,\,\,\text{on $\left[\sum_{i=0}^{n-2}\frac{1}{3^i},\sum_{i=0}^{n-1}\frac{1}{3^i}\right]=D_n.$}
\] 

We shall prove that $a$ satisfies condition (\ref{I}) with $\nu=1/2$, but it fails (\ref{O}) with the same $\nu$.
\begin{proposition}
There exists $C>0$ such that for any interval $I$,
\begin{equation}
|a^{-1}(I)|=\left|\left\{v: a(v) \in I \right\}\right| \leq C|I|^{1/2}. 
\end{equation}
\end{proposition}

\begin{proof}

Consider an interval $I=\left[\sum_{i=0}^{n-1}\frac{1}{3^{2i}}-p_2,\sum_{i=0}^{n-1}\frac{1}{3^{2i}}-p_1\right]=[c,d]$ inside some $a(D_n)$, where $0\leq p_1< p_2\leq \frac{1}{3^{n-1}}.$ So $|I|=p_2-p_1.$
Denote the pre-image of $c$ and $d$ by $v_2$ and $v_1$ respectively. Then we have for each $k=1,2$,
$$
a_n(v_k)=\sum_{i=0}^{n-2} \frac{1}{3^{2i}}+\frac{1}{3^{2(n-1)}} a_1\left((v_k-\sum_{i=0}^{n-2}\frac{1}{3^i})3^{n-1}\right)=\sum_{i=0}^{n-1}\frac{1}{3^{2i}}-p_k.
$$
So
$$
a_n^{-1}\left(\sum_{i=0}^{n-1}\frac{1}{3^{2i}}-p_1\right)=v_k=\sum_{i=0}^{n-1}\frac{1}{3^i}-\sqrt{p_k}.
$$

We therefore have 
$$
|a^{-1}(I)|=\sqrt{p_2}-\sqrt{p_1}\leq \sqrt{p_2-p_1}=|I|^{1/2}.
$$
If $I=[c,d]\subset a(\cup_{i=m_1}^{m_2} D_i)$, separate $I$ into three sub-intervals: $I=I_1\cup I_2 \cup I_3,$ where $I_1=\left[c,\sum_{i=0}^{m_1-1}\frac{1}{3^{2i}}\right]$, $I_2=\left[\sum_{i=0}^{m_1-1}\frac{1}{3^{2i}},\sum_{i=0}^{m_2-2}\frac{1}{3^{2i}}\right]$ and $I_3=\left[\sum_{i=0}^{m_2-2}\frac{1}{3^{2i}},d\right].$ The above case applies to $I_1$ and $I_3$, so $|a^{-1}(I_1)|\leq |I_1|^{1/2}$ and $|a^{-1}(I_3)|\leq |I_3|^{1/2}.$

For $I_2,$ we have 
$$
|I_2|=\sum_{m_1}^{m_2-2}\frac{1}{3^{2i}}=\frac{9}{8}\frac{1}{3^{2m_1}}\left[1-\left(\frac{1}{9}\right)^{m_2-m_1-1}\right].
$$

And
\begin{align*}
|a^{-1}(I_2)|^2&=\left(\sum_{m_1}^{m_2-2}\frac{1}{3^i}\right)^2=\frac{9}{4}\frac{1}{3^{2m_1}}\left[1-\left(\frac{1}{3}\right)^{m_2-m_1-1}\right]^2\\
&\leq \frac{9}{4}\frac{1}{3^{2m_1}}\left[1-2\left(\frac{1}{9}\right)^{m_2-m_1-1}+\left(\frac{1}{9}\right)^{m_2-m_1-1}\right]\\
&=2|I_2|.
\end{align*}
So 
$$
|a^{-1}(I_2)|\leq 2^{1/2}|I_2|^{1/2}.
$$

Notice that this inequality is still true when $m_2$ goes to infinity, so there are no issues near the right end point.

Combining the three inequalities we get
$$
|a^{-1}(I)|=\sum_{i=1}^3| a^{-1}(I_i)|\leq  2^{1/2}\sum_{i=1}^3|I_i|^{1/2} \leq 6^{1/2} \left(\sum_{i=1}^3|I_i|\right)^{1/2}=6^{1/2} |I|^{1/2}.
$$
\end{proof}

\begin{proposition}
There exists a sequence of set $\mathcal{O}^m$ such that
$$
\frac{|a^{-1}(\mathcal{O}^m)|}{|\mathcal{O}^m|^{1/2}}\rightarrow \infty \,\,\,\,\,\,\,\,\text{as $m\rightarrow \infty.$}
$$

\end{proposition}

\begin{proof}
Let
$$
\mathcal{O}^m=\cup_{n=1}^m I_n,
$$
where $I_n=\left[\sum_{i=0}^{n-1}\frac{1}{3^{2i}}-\frac{1}{3^{2(m-1)}},\sum_{i=0}^{n-1}\frac{1}{3^{2i}}\right]$ for all $1\leq n\leq m$.

So 
$$
|I_n|=\frac{1}{3^{2(m-1)}}\,\,\,\,\,\,\text{for all $1\leq n\leq m$,}
$$
and 
$$
|a^{-1}(I_n)|=|I_n|^{1/2}=\frac{1}{3^{m-1}}\,\,\,\,\,\,\text{for all $1\leq n\leq m$.}
$$

Therefore,
$$
\frac{|a^{-1}(\mathcal{O}^m)|}{|\mathcal{O}^m|^{1/2}}=\frac{\frac{m}{3^{m-1}}}{\left(\frac{m}{3^{2(m-1)}}\right)^{1/2}}=\sqrt{m}\rightarrow \infty \,\,\,\,\,\,\,\,\text{as $m\rightarrow \infty.$}
$$
\end{proof}

\end{document}